\documentclass[a4, 12pt]{amsart}
\usepackage{fullpage}
\usepackage[title]{appendix}
\usepackage{graphicx}
\usepackage{subcaption}
\usepackage{amsthm}
\usepackage{mathrsfs} 
\usepackage{amsfonts}
\usepackage{amssymb}
\usepackage{diagbox}
\usepackage{mathtools}
\usepackage{longtable} 
\usepackage{MnSymbol}
\usepackage{indentfirst}
\usepackage{listings,xcolor}
\usepackage{MnSymbol}
\usepackage{geometry}
\usepackage[colorlinks=true,linkcolor=blue]{hyperref}

\newtheorem{theorem}{Theorem}[section]
\newtheorem{proposition}[theorem]{Proposition}
\newtheorem{lemma}[theorem]{Lemma}
\newtheorem{corollary}[theorem]{Corollary}
\newtheorem{conjecture}[theorem]{Conjecture}

\theoremstyle{definition}

\newtheorem{remark}[theorem]{Remark}

\numberwithin{equation}{section}

\newcommand{\bZ}{\mathbb{Z}}

\newcommand{\qbm}[2]{\genfrac{[}{]}{0pt}{}{#1}{#2}}

\begin{document}

\title[Positivity and tails of Jacobi theta series]
{Positivity and tails of Jacobi theta series}

\author{Nian Hong \textsc{Zhou}}

\address{
School of Mathematics and Statistics, The Center for Applied Mathematics of Guangxi,
Guangxi Normal University, Guilin, 541006, Guangxi, PR China
}
\email{nianhongzhou@outlook.com; nianhongzhou@gxnu.edu.cn}

\thanks{
This work was partially supported by the National Natural Science Foundation of China
(No. 12301423).
}


\begin{abstract}
Using elementary $q$-series manipulations, we establish a positivity property for the tails of the Jacobi theta series. Specifically, for integers $k\ge 1$ and $n\ge 0$, define
\[
\sum_{n\ge0}\sum_{m\in\bZ}J_{k,n}(m)z^m q^{n}
=
\frac{(-1)^k q^{-\binom{k+1}{2}}}{(z)_{\infty}(q/z)_\infty} 
\sum_{j\ge k}(-1)^jq^{\binom{j+1}{2}}z^{-j}(1-z^{2j+1}),
\]
where $(a)_\infty:=\prod_{n\ge0}(1-aq^n)$ denotes the $q$-shifted factorial. 
We prove that for all integers $k\ge 1$ and $n\ge 0$, the coefficients $J_{k,n}(m)$ are positive for all integers $-(k+n)\le m\le k+n$.
\end{abstract}

\maketitle

\section{Introduction}

Throughout this paper,
let the \textit{$q$-shifted factorial} (cf.\  \cite{MR2128719})
be defined by
$$
(a)_\infty:=(a;q)_\infty:=\prod_{j\ge 0}(1-aq^j),\quad\text{and}\quad
(a)_c:=(a;q)_c:=\frac{(a;q)_\infty}{(aq^c;q)_\infty},
$$
for any indeterminate $a$ and $c\in\bZ$. For nonnegative integers $n$ and $k$, the
\textit{$q$-binomial coefficient} is defined as
\begin{equation*}
\qbm{n}{k}_q:=
\frac{(q;q)_n}{(q;q)_k(q;q)_{n-k}}.
\end{equation*}

Euler's pentagonal number theorem
(cf.\ \cite[Equation~(8.10.10)]{MR2128719}) is stated as
\begin{equation}\label{PNT}
(q;q)_\infty=\sum_{j\ge 0}(-1)^jq^{j(3j+1)/2}(1-q^{2j+1}),
\end{equation}
which is one of the most famous $q$-series identities and plays an important role in the theory of partitions. In 2012, Andrews and Merca~\cite{MR2946378} gave an explicit expansion for
the averaged truncation of Euler's pentagonal number series appearing in \eqref{PNT}.
In particular, they \cite[Lemma~1.2]{MR2946378} proved that
\begin{align*}
\frac{1}{(q;q)_\infty}\sum_{0\le j< k}(-1)^jq^{j(3j+1)/2}(1-q^{2j+1})=1+(-1)^{k-1}\sum_{n\ge k} \frac{q^{\binom{k}{2}+(k+1)n}}{(q;q)_n}
  \qbm{n-1}{k-1}_q.
\end{align*}
Consequently, according to the parity of $k$, the nonconstant coefficients
on the right-hand side of the above are either all nonnegative or all
nonpositive.

\medskip

A natural generalization of Euler's pentagonal number theorem is Jacobi's
triple product identity
\begin{equation}\label{JTP}
(z)_\infty(z^{-1}q)_\infty(q)_\infty
=
\sum_{n\in\bZ}(-1)^nq^{\binom{n}{2}}z^n.
\end{equation}
Indeed, \eqref{JTP} reduces to \eqref{PNT} after replacing $q$ by $q^3$
and setting $z=q$ or $z=q^2$. The following result on averaged truncations of Jacobi's triple product identity was originally conjectured by Andrews and Merca~\cite[Question~(2)]{MR2946378} and independently by Guo and Zeng~\cite[Conjecture~6.1]{MR3007145}, and has since been established by multiple proofs.

\begin{theorem}\label{AMthm}
For positive integers $m,k,R,S$ with $1\le S<R/2$, the coefficient of $q^m$ in
\begin{equation*}
\frac{(-1)^{k-1}}{(q^S;q^R)_\infty(q^{R-S};q^R)_\infty(q^{R};q^{R})_{\infty}}
\sum_{0\le n<k}(-1)^nq^{\binom{n+1}{2}R-nS}(1-q^{(2n+1)S})
\end{equation*}
is non-negative.
\end{theorem}

The first proofs of Theorem~\ref{AMthm} were given independently by 
Mao~\cite{MR3280682} and Yee~\cite{MR3280681} in 2015, using algebraic 
and combinatorial methods respectively. A combinatorial proof was later 
provided by He, Ji, and Zang~\cite{MR3398853}. More recently, 
Wang and Yee~\cite[Theorem~2.3]{MR3937794} and Schlosser and the 
author~\cite[Theorem~1.1]{SZ20231} established different $q$-series 
identities that imply this result. 

\medskip

In~\cite{MR4309312}, Merca proposes the following conjecture on truncations of Jacobi's triple product identity, which refines the above Theorem ~\ref{AMthm}.

\begin{conjecture}\label{conjt1}
Let $1\le S<R/2$ and $k\ge1$. Then the theta series
\[
\frac{(-1)^k}
{(q^S;q^R)_\infty(q^{R-S};q^R)_\infty}
\sum_{j\ge k}(-1)^jq^{j(j+1)R/2-jS}
(1-q^{(2j+1)S})
\]
has non-negative coefficients.
\end{conjecture}
In \cite{MR4768277}, the author proposed a positivity
conjecture involving a two-variable theta series, which further refines
Conjecture~\ref{conjt1}. To state its positivity part, let $k,d\ge1$ and
$n\ge0$ be integers, and define Laurent polynomials
$\mathscr{L}_{k,n}^{d}(z)$ by
\[
\sum_{n\ge0}\mathscr{L}_{k,n}^{d}(z)q^{n}
=
\frac{(-1)^kq^{-k(k+1)/2}}{(z)_{d+1}(z^{-1}q)_d}
\sum_{j\ge k}(-1)^jq^{j(j+1)/2}z^{-j}(1-z^{2j+1}).
\]
One part of the author's positivity conjecture is stated as follows.
\begin{conjecture}\label{conjt20}
For all integers $k\ge1$, $d\ge2$, and $n\ge0$, the Laurent polynomial
\[
\mathscr{L}_{k,n}^{d}(z)
=
\sum_{-n-k\le m\le n+k}J_{k,n}^{d}(m)z^m
\]
has positive coefficients.

\end{conjecture}
\begin{remark} 
With $q$ and $z$ replaced by $q^R$ and $q^S$, respectively, we observe that 
$(z)_{d+1}(q/z)_d$ becomes $(q^S;q^R)_{d+1}(q^{R-S};q^R)_d$,
which appears in $(q^S;q^R)_\infty (q^{R-S};q^R)_\infty$.
Therefore, if Conjecture \ref{conjt20} holds for some $d\in \bZ_{\ge 2}\cup\{\infty\}$, then it implies Conjecture \ref{conjt1}.
\end{remark}

Very recently, Ding and Sun~\cite{arXiv:2606.27507} announced a combinatorial proof of Conjecture~\ref{conjt1}. Motivated by their work, we use elementary $q$-series manipulations to prove Conjecture~\ref{conjt20} in the limiting case $d=\infty$. Our proof is inspired by the ideas of Ding and Sun~\cite{arXiv:2606.27507} as well as the author's recent work~\cite{arXiv:2606.13274} on monotonicity of rank functions for concave compositions.

\medskip

For any integer $a\ge1$ and indeterminate $x$, define
\begin{equation}\label{eqdfg}
f_a(x)
=
\sum_{i\ge0}
\frac{q^{i^2+i(a-1)}x^i}{(q)_i(x)_i}\quad \text{and}
\quad
g_a(x)
=
\sum_{i\ge0}
\frac{q^{i^2+ia}x^i}{(q)_i(x)_{i+1}}.
\end{equation}
Our main result is the following identity, from which the case $d=\infty$ of Conjecture \ref{conjt20} follows immediately, and consequently Merca's Conjecture \ref{conjt1} as well.
\begin{theorem}\label{main}
For every integer $k\ge1$, we have
\begin{align}
&
\frac{(-z)^k q^{-\binom{k+1}{2}}}
{(z)_\infty(z^{-1}q)_\infty}
\sum_{j\ge k}(-z)^{-j}q^{\binom{j+1}{2}}(1-z^{2j+1})
\nonumber\\
&\quad =
\frac{1+z+\cdots+z^{2(k-1)}}
{(qz)_\infty(z^{-1}q)_k}
f_1(z^{-1}q^{k+1})+
\frac{z^{2k-1}+z^{2k}}{(z^{-1}q)_k(qz)_k}
\sum_{\substack{r\ge k\\ r\equiv k\pmod 2}}
\frac{z^{2(r-k)} f_2(z^{-1}q^{r+1})}
{(zq^{k+1})_{r-k}(zq^{r+2})_\infty}
\nonumber\\
&\qquad
+
\frac{1+z}{(z^{-1}q)_k}
\sum_{\substack{r\ge k\\ r\equiv k\pmod 2}}
\frac{
z^{2r-2}q^{r+2}g_2(z^{-1}q^{r+1})
}
{(zq)_{r-1}(zq^{r+1})_\infty}.
\label{main-id}
\end{align}
\end{theorem}
Consequently, we have the following corollary.
\begin{corollary}\label{mcor}
For every integer $k\ge1$, we have
\begin{align*}
\sum_{n\ge0} \sum_{m\in\bZ}J_{k,n}(m)z^m q^{n}:=&\frac{(-1)^k q^{-\binom{k+1}{2}}}
{(z)_\infty(z^{-1}q)_\infty}
\sum_{j\ge k}(-z)^{-j}q^{\binom{j+1}{2}}(1-z^{2j+1})\\
=&\frac{z^{-k}(1+z+\cdots+z^{2k})}
{(1-qz)(1-z^{-1}q)}+z^{-k}{\sf N}_k(z,q),
\end{align*}
where ${\sf N}_k(z,q)$ has nonnegative coefficients in its Laurent series expansion in $z$ and $q$. 
In particular, for all integers $k\ge 1$ and $n\ge 0$, the coefficients $J_{k,n}(m)$ are positive for all integers $m$ with $-(k+n)\le m\le k+n$; that is, Conjecture~\ref{conjt20} holds in the limiting case $d=\infty$.
\end{corollary}

\section{Proof of the main theorem}\label{sec2}

We first record an elementary relation between the series $f_a$ and $g_a$ defined by \eqref{eqdfg}.
It is a special case of \cite[Proposition~2.1]{arXiv:2606.13274}. For the
reader's convenience, we include the proof.

\begin{lemma}\label{lemm}
For every $a\ge1$ and every $\beta$, we have
\[
f_a(x)-\beta g_a(x)
=
(1-\beta)f_{a+1}(x)+xq^a(1-q^{-a}\beta)g_{a+1}(x).
\]
\end{lemma}

\begin{proof}
Using \eqref{eqdfg}, we obtain
\begin{align*}
f_a(x)
&=
\sum_{i\ge0}
\frac{q^{i^2+i(a-1)}x^i(1-q^i+q^i)}
{(q)_i(x)_i}
\\
&=
\sum_{i\ge1}
\frac{q^{i^2+i(a-1)}x^i}{(q)_{i-1}(x)_i}
+
\sum_{i\ge0}
\frac{q^{i^2+ia}x^i}{(q)_i(x)_i}
\\
&=
xq^a g_{a+1}(x)+f_{a+1}(x).
\end{align*}
Similarly,
\begin{align*}
g_a(x)
&=
\sum_{i\ge0}
\frac{q^{i^2+ia}x^i(1-xq^i+xq^i)}
{(q)_i(x)_{i+1}}
\\
&=
\sum_{i\ge0}
\frac{q^{i^2+ia}x^i}{(q)_i(x)_i}
+
\sum_{i\ge0}
\frac{q^{i^2+i(a+1)}x^{i+1}}{(q)_i(x)_{i+1}}
\\
&=
f_{a+1}(x)+xg_{a+1}(x).
\end{align*}
Subtracting $\beta g_a(x)$ from $f_a(x)$ gives
\[
f_a(x)-\beta g_a(x)
=
(1-\beta)f_{a+1}(x)+x(q^a-\beta)g_{a+1}(x),
\]
which is the desired identity.
\end{proof}

The next lemma connects a \emph{partial theta series} with a \emph{basic
hypergeometric series}. This connection is one of the ingredients that makes
the limiting case $d=\infty$ accessible by using Lemma \ref{lemm}.

\begin{lemma}\label{lem1}
For every $x$, we have
\[
\frac{1}{(x)_\infty}
\sum_{n\ge0}(-x)^nq^{\binom{n}{2}}
=
\sum_{n\ge0}
\frac{q^{n^2}x^n}{(q)_n(x)_n}.
\]
\end{lemma}

\begin{proof}
We use Heine's second transformation
\[
\sum_{n\ge0}
\frac{(a)_n(b)_n z^n}{(q)_n(c)_n}
=
\frac{(c/b)_\infty(bz)_\infty}{(c)_\infty(z)_\infty}
\sum_{n\ge0}
\frac{(abz/c)_n(b)_n}{(q)_n(bz)_n}
\left(\frac{c}{b}\right)^n;
\]
see \cite[Appendix, (III.2)]{MR2128719}. Put $z=xq/ab$ and $c=x$. Then
\[
\sum_{n\ge0}
\frac{(a)_n(b)_n(xq/ab)^n}{(q)_n(x)_n}
=
\frac{(x/b)_\infty(xq/a)_\infty}{(x)_\infty(xq/ab)_\infty}
\sum_{n\ge0}
\frac{(q)_n(b)_n}{(q)_n(xq/a)_n}
\left(\frac{x}{b}\right)^n.
\]
Letting $a,b\to\infty$ gives the assertion.
\end{proof}

The following decomposition is equivalent to a lemma of Ding and Sun
\cite[Lemma~3.2]{arXiv:2606.27507}.

\begin{lemma}\label{lem2}
For $r\ge1$, define
\[
h_r(z):=(-1)^r\sum_{j\ge r}(-z)^{-j}q^{j(j+1)/2}.
\]
Then, for every integer $k\ge1$,
\begin{align*}
&
\frac{(-1)^k}
{(z)_\infty(q/z)_\infty}
\sum_{j\ge k}(-z)^{-j}q^{\binom{j+1}{2}}(1-z^{2j+1})
\\
&\quad =
\frac{z^{-k}(1+z+\cdots+z^{2(k-1)})}
{(qz)_\infty(q/z)_\infty}
h_k(z)
\\
&\qquad
+
\frac{1+z}{(qz)_\infty(q/z)_\infty}
\sum_{\substack{r\ge k\\ r\equiv k\pmod 2}}
z^{2r-1}\bigl(h_r(z)-z^2h_{r+1}(z)\bigr).
\end{align*}
\end{lemma}

\begin{proof}
We compute
\begin{align*}
&
\frac{(-1)^k}{1-z}
\sum_{j\ge k}(-z)^{-j}q^{\binom{j+1}{2}}(1-z^{2j+1})
\\
&=
\sum_{j\ge k}
(-1)^{j-k}q^{j(j+1)/2}z^{-j}
\sum_{0\le r\le2j}z^r
\\
&=
h_k(z)\sum_{0\le r\le2(k-1)}z^r
+
(1+z)
\sum_{j\ge k}
(-1)^{j-k}q^{j(j+1)/2}z^{-j}
\sum_{k\le r\le j}z^{2r-1}
\\
&=
h_k(z)\sum_{0\le r\le2(k-1)}z^r
+
(1+z)
\sum_{r\ge k}
(-1)^{r-k}z^{2r-1}
\sum_{j\ge r}
(-1)^{j-r}q^{j(j+1)/2}z^{-j}
\\
&=
h_k(z)\sum_{0\le r\le2(k-1)}z^r
+
(1+z)
\sum_{r\ge k}
(-1)^{r-k}z^{2r-1}h_r(z).
\end{align*}
Equivalently,
\begin{align*}
\frac{(-1)^k}{1-z}
\sum_{j\ge k}(-z)^{-j}q^{\binom{j+1}{2}}(1-z^{2j+1})
=&(1+z+\cdots+z^{2(k-1)})h_k(z)\\
&+
z^{2k-1}(1+z)
\sum_{s\ge0}z^{4s}
\bigl(h_{k+2s}(z)-z^2h_{k+2s+1}(z)\bigr).
\end{align*}
Dividing by $(qz)_\infty(q/z)_\infty$ proves the result.
\end{proof}

The next proposition expresses $h_k(z)$ and the difference
$h_k(z)-z^2h_{k+1}(z)$ in terms of the series $f_a$ and $g_a$.

\begin{proposition}\label{prop-hk}
For every integer $k\ge1$, we have
\[
h_k(z)
=
z^{-k}q^{\binom{k+1}{2}}
(z^{-1}q^{k+1})_\infty f_1(z^{-1}q^{k+1}),
\]
and
\begin{align*}
&
\frac{z^kq^{-\binom{k+1}{2}}}
{(z^{-1}q^{k+1})_\infty}
\bigl(h_k(z)-z^2h_{k+1}(z)\bigr)
\\
&\quad =
(1-zq^{k+1})f_2(z^{-1}q^{k+1})
+
z^{-1}q^{k+2}(1-zq^k)g_2(z^{-1}q^{k+1}).
\end{align*}
\end{proposition}

\begin{proof}
By definition,
\begin{align*}
h_k(z)
&=
\sum_{j\ge k}
(-1)^{j-k}q^{\binom{j+1}{2}}z^{-j}
\\
&=
z^{-k}q^{\binom{k+1}{2}}
\sum_{j\ge0}
(-1)^jq^{\binom{j}{2}}(z^{-1}q^{k+1})^j.
\end{align*}
Applying Lemma~\ref{lem1} with $x=z^{-1}q^{k+1}$ gives
\[
h_k(z)
=
z^{-k}q^{\binom{k+1}{2}}
(z^{-1}q^{k+1})_\infty
\sum_{n\ge0}
\frac{q^{n^2}(z^{-1}q^{k+1})^n}
{(q)_n(z^{-1}q^{k+1})_n},
\]
which is the first identity.

Replacing $k$ by $k+1$ in the preceding formula gives
\begin{align*}
h_{k+1}(z)
&=
z^{-k-1}q^{\binom{k+2}{2}}
(z^{-1}q^{k+2})_\infty
\sum_{n\ge0}
\frac{q^{n^2}(z^{-1}q^{k+2})^n}
{(q)_n(z^{-1}q^{k+2})_n}
\\
&=
z^{-k-1}q^{\binom{k+1}{2}+k+1}
(z^{-1}q^{k+1})_\infty
\sum_{n\ge0}
\frac{q^{n^2+n}(z^{-1}q^{k+1})^n}
{(q)_n(z^{-1}q^{k+1})_{n+1}}.
\end{align*}
Therefore, by using the definition \eqref{eqdfg} of $f_a$ and $g_a$,
\begin{align*}
\frac{h_k(z)-z^2h_{k+1}(z)}
{z^{-k}q^{\binom{k+1}{2}}(z^{-1}q^{k+1})_\infty}
 =
f_1(z^{-1}q^{k+1})-zq^{k+1}g_1(z^{-1}q^{k+1}).
\end{align*}
Applying Lemma~\ref{lemm} with $a=1$, $x=z^{-1}q^{k+1}$, and
$\beta=zq^{k+1}$ yields
\[
f_1(z^{-1}q^{k+1})-zq^{k+1}g_1(z^{-1}q^{k+1})
=
(1-zq^{k+1})f_2(z^{-1}q^{k+1})
+
z^{-1}q^{k+2}(1-zq^k)g_2(z^{-1}q^{k+1}),
\]
as required.
\end{proof}
We now prove Theorem \ref{main} and Corollary \ref{mcor}.
\begin{proof}[Proof of Theorem~\ref{main}]
Substituting the two identities in Proposition~\ref{prop-hk} into
Lemma~\ref{lem2}, and using
\[
(z^{-1}q)_\infty=(z^{-1}q)_k(z^{-1}q^{k+1})_\infty,
\]
gives
\begin{align*}
&
\frac{(-z)^k q^{-\binom{k+1}{2}}}
{(z)_\infty(z^{-1}q)_\infty}
\sum_{j\ge k}(-z)^{-j}q^{\binom{j+1}{2}}(1-z^{2j+1})
\\
&=
\frac{1+z+\cdots+z^{2(k-1)}}
{(qz)_\infty(z^{-1}q)_k}
f_1(z^{-1}q^{k+1})+
\frac{(1+z)z^{2k-1}}{(z^{-1}q)_k}
\sum_{\substack{r\ge k\\ r\equiv k\pmod 2}}
\frac{
z^{2(r-k)}f_2(z^{-1}q^{r+1})}
{(zq)_{r}(zq^{r+2})_\infty}
\\
&\quad
+
\frac{1+z}{(z^{-1}q)_k}
\sum_{\substack{r\ge k\\ r\equiv k\pmod 2}}
\frac{
z^{2r-2}q^{r+2}g_2(z^{-1}q^{r+1})
}
{(zq)_{r-1}(zq^{r+1})_\infty},
\end{align*}
which completes the proof.
\end{proof}
\begin{proof}[Proof of Corollary~\ref{mcor}]
Our main identity \eqref{main-id} can be rewritten as follows:
\begin{align*}
\frac{(-z)^k q^{-\binom{k+1}{2}}}
{(z)_\infty(z^{-1}q)_\infty}
\sum_{j\ge k}(-z)^{-j}q^{\binom{j+1}{2}}(1-z^{2j+1})
=\frac{1+z+\cdots+z^{2k}}
{(1-qz)(1-z^{-1}q)}+{\sf N}_k(z,q),
\end{align*}
where
\begin{align*}
{\sf N}_k(z,q)=&\frac{1+z+\cdots+z^{2(k-1)}}
{(1-zq)(1-z^{-1}q)}\left(
\frac{f_1(z^{-1}q^{k+1})}{(q^2z)_\infty(z^{-1}q^2)_{k-1}}-1\right)\\
&+\frac{(1+z)z^{2k-1}}{(1-z^{-1}q)(1-zq)}\left(
\frac{f_2(z^{-1}q^{k+1})}
{(z^{-1}q^2)_{k-1}(zq^2)_{k-1}(zq^{k+2})_\infty}-1\right)
\\
&+\frac{1+z}{(z^{-1}q)_k}
\left(\sum_{\substack{r> k\\ r\equiv k\pmod 2}}
\frac{
z^{2r-1}f_2(z^{-1}q^{r+1})}
{(zq)_{r}(zq^{r+2})_\infty}
+
\sum_{\substack{r\ge k\\ r\equiv k\pmod 2}}
\frac{
z^{2r-2}q^{r+2}g_2(z^{-1}q^{r+1})
}
{(zq)_{r-1}(zq^{r+1})_\infty}\right)
\end{align*}
has nonnegative coefficients. Here we have used the fact that $f_1$, $f_2$ and $g_2$ have
nonnegative coefficients after expanding the reciprocal $q$-shifted
factorials, and the constant terms of $f_1$ and $f_2$ are equal to $1$.  Thus,
\begin{align*}
\sum_{n\ge 0}\sum_{m\in\bZ}J_{k,n}(m)z^m q^{n}
&=
\frac{(-1)^kq^{-\binom{k+1}{2}}}{(z)_{\infty}(q/z)_\infty}
\sum_{j\ge k}(-1)^jq^{\binom{j+1}{2}}z^{-j}(1-z^{2j+1})\\
&=
\frac{z^{-k}(1+z+\cdots+z^{2k})}
{(1-qz)(1-z^{-1}q)}+z^{-k}{\sf N}_k(z,q).
\end{align*}
Note that
\begin{align*}
\frac{z^{-k}(1+z+\cdots+z^{2k})}
{(1-qz)(1-z^{-1}q)}&=\sum_{n\ge 0}q^nz^{-k}\sum_{\substack{h+l=n\\ h,l\ge 0}}\sum_{0\le r\le 2k}z^{h-l+r}\\
&=\sum_{n\ge 0}q^n z^{-n-k}\sum_{0\le h\le n} \sum_{0\le r\le 2k} q^{2h+r}\\
&=\sum_{n\ge 0}q^n z^{-n-k} \sum_{0\le t\le 2(n+k)} z^t\sum_{\substack{0\le h\le n\\ 0\le t-2h=r\le 2k}}1,
\end{align*}
and for all $0\le t\le 2(n+k)$,
$$\sum_{\substack{0\le h\le n\\ 0\le t-2h=r\le 2k}}1\ge 1,$$
we immediately obtain $J_{k,n}(m)\ge 1$ for all $-k-n\le m\le k+n$, which completes the proof of the corollary.
\end{proof}


\begin{thebibliography}{1}

\bibitem{MR2946378}
G.~E. Andrews and M. Merca.
\newblock The truncated pentagonal number theorem.
\newblock {\em J. Combin. Theory Ser. A}, 119(8):1639--1643, 2012.

\bibitem{MR2128719}
G. Gasper and M. Rahman.
\newblock {\em Basic hypergeometric series}, volume~96 of {\em Encyclopedia of
  Mathematics and its Applications}.
\newblock Cambridge University Press, Cambridge, second edition, 2004.
\newblock With a foreword by Richard Askey.


\bibitem{arXiv:2606.27507}
C.~Ding and L.~Sun,
\emph{A combinatorial proof for the positivity of the normalized Jacobi triple product tails},
preprint, arXiv:2606.27507.

\bibitem{MR3007145}
V. J.~W. Guo and J. Zeng.
\newblock Two truncated identities of {G}auss.
\newblock {\em J. Combin. Theory Ser. A}, 120(3):700--707, 2013.

\bibitem{MR3398853}
T.~Y. He, K.~Q. Ji, and W. J.~T. Zang.
\newblock Bilateral truncated {J}acobi's identity.
\newblock {\em European J. Combin.}, 51:255--267, 2016.

\bibitem{MR3280682}
R. Mao.
\newblock Proofs of two conjectures on truncated series.
\newblock {\em J. Combin. Theory Ser. A}, 130:15--25, 2015.

\bibitem{MR4309312}
M. Merca.
\newblock Truncated theta series and {R}ogers-{R}amanujan functions.
\newblock {\em Exp. Math.}, 30(3):364--371, 2021.

\bibitem{SZ20231}
M.~J. Schlosser and N.~H. Zhou.
\newblock Expansions of averaged truncations of basic hypergeometric series.
\newblock {\em Proc. Amer. Math. Soc.}, 152(11):4659--4673, 2024.

\bibitem{MR3937794}
C. Wang and A.~J. Yee.
\newblock Truncated {J}acobi triple product series.
\newblock {\em J. Combin. Theory Ser. A}, 166:382--392, 2019.

\bibitem{MR3280681}
A.~J. Yee.
\newblock A truncated {J}acobi triple product theorem.
\newblock {\em J. Combin. Theory Ser. A}, 130:1--14, 2015.




\bibitem{MR4768277}
N.~H. Zhou.
\newblock Positivity and tails of pentagonal number series.
\newblock {\em J. Combin. Theory Ser. A}, 208:Paper No. 105933, 21, 2024.


\bibitem{arXiv:2606.13274}
N.~H. Zhou,
\emph{Monotonicity of rank functions for concave compositions},
preprint, arXiv:2606.13274.

\end{thebibliography}

\end{document}